\documentclass[10pt,a4paper,reqno,oneside]{smfart}

\usepackage{leftidx}
\usepackage{appendix}
\usepackage{tikz}\usetikzlibrary{decorations.markings,matrix,arrows}
\usepackage{amsfonts}
\usepackage{amsthm}
\usepackage[T1]{fontenc}
\usepackage[koi8-r, utf8]{inputenc}
\usepackage[mathscr]{eucal}
\usepackage{setspace}
\usepackage{amssymb,latexsym,amsmath,amsfonts}
\usepackage{graphicx}
\usepackage[russian,french]{babel}
\usepackage{layout}
\usepackage{enumerate}
\usepackage{indentfirst}
\usepackage[varg]{pxfonts}
\usepackage{amscd}
\usepackage{xspace}
\usepackage{eufrak}
\usepackage{hyperref}
\usepackage{nameref}
\usepackage{calrsfs}
\usepackage{mathtools}

\newtheoremstyle{exostyle} 
{\topsep}
{\topsep}
{}
{}
{\bfseries}
{.}
{ }
{\thmname{#1}\thmnumber{ #2}\thmnote{. \normalfont{\textit{#3}}}}

\theoremstyle{exostyle} 
 
\theoremstyle{plain}
\newtheorem{thm}{Theorem}[section]
\newtheorem*{thm*}{Theorem}

\newtheorem{lem}[thm]{Lemma}

\newtheorem{pro}[thm]{Proposition}
\newtheorem{cor}[thm]{Corollary}

\theoremstyle{definition}
\newtheorem{Def}[thm]{Definition}

\newtheorem{rem}[thm]{Remark}

\theoremstyle{remark}

\newcommand{\wt}{\widetilde}

\newcommand{\bP}{\mathbb{P}}

\newcommand{\bC}{\mathbb{C}}

\newcommand{\bQ}{\mathbb{Q}}
\newcommand{\bZ}{\mathbb{Z}}

\newcommand{\gl}{\lambda}

\newcommand{\gO}{\Omega}
\newcommand{\ga}{\alpha}
\newcommand{\gS}{\Sigma}
\newcommand{\gb}{\beta}

\newcommand{\gs}{\sigma}

\newcommand{\cC}{\mathcal{C}}

\newcommand{\cF}{\mathcal{F}}

\newcommand{\cO}{\mathcal{O}}

\newcommand{\reg}{\mathrm{reg}}

\newcommand{\ol}{\overline}

\newcommand{\colonec}{\mathrel{:=}}

\newcommand{\cusp}{\mathrm{cusp}}
\newcommand{\Defo}{\mathrm{Def}}
\newcommand{\supp}{\mathrm{supp}}
\newcommand{\univ}{\mathrm{univ}}

\newcommand{\Td}{\mathrm{Td}}

\newcommand{\chchar}{\mathrm{ch}}

\newcommand{\CH}{\mathrm{CH}}
\newcommand{\pr}{\mathrm{pr}}

\newcommand{\Pic}{\mathrm{Pic}}

\newcommand{\hhom}{\mathrm{hom}}

\newcommand{\tors}{\mathrm{tors}}

\newcommand{\coker}{\mathrm{coker}}

\newcommand{\jac}[1]{\mathrm{jac}\left( #1 \right)}
\newcommand{\dr}{\partial}

\renewcommand{\(}{\left(}
\renewcommand{\)}{\right)}

\newcommand{\dto}{\dashrightarrow}

\makeatletter
\let\orgdescriptionlabel\descriptionlabel
\renewcommand*{\descriptionlabel}[1]{%
  \let\orglabel\label
  \let\label\@gobble
  \phantomsection
  \edef\@currentlabel{#1}%
  \let\label\orglabel
  \orgdescriptionlabel{#1}%
}
\makeatother
\tikzset{node distance=2cm, auto}

\numberwithin{equation}{section}

\title{Singularities of elliptic curves in $K3$ surfaces and the Beauville-Voisin zero-cycle} 

\author{Hsueh-Yung Lin}
\address{Centre de Math\'ematiques Laurent Schwartz, 91128 Palaiseau C\'edex, France}

\begin{document}

\maketitle

\begin{altabstract}
Under some hypotheses on the singular type of the one-parameter family of elliptic curves in a primitively polarized  $K3$ surface $S$ determined by its polarization  (which is expected to be true for a very general polarized $K3$ surface), we give a more geometric proof of the fact that the second Chern class of $S$ is equal to $24 \cdot o_S$ in the Chow group of $0$-cycles where $o_S$ is the Beauville-Voisin canonical $0$-cycle.
\end{altabstract}

\section{Introduction}

For any smooth complex projective surface $S$, D. Mumford~\cite{Mumzerocycle} showed that as long as  $H^0(S, \gO_S^2)$ is nontrivial, the Chow group of  $0$-cycles $\CH_0(S)$ is infinite dimensional in the sense of Roitman, which is equivalent to the property that for any smooth projective variety $V$, not necessarily connected, and any correspondence $\Gamma \subset V \times S$, the Chow group of  $0$-cycles of degree zero ${\CH_0(S)}_\hhom$ is never contained in $\left\{ \Gamma_*[x] - [x_0] \mid x \in V \right\}$ for any $x_0 \in X$~\cite{Roitman}. 

Recall that a  projective $K3$ surface $S$ is a smooth projective surface with trivial canonical  line bundle $K_S = \gO_S^2$ and vanishing $H^1(S,\cO_S)$. Since $H^0(S, \gO_S^2) \ne \{0\}$ for any $K3$ surface $S$, Mumford's result implies that $\CH_0(S)$ is very large whenever $S$ is a $K3$ surface.

In contrast to Mumford's result on the Chow group of $0$-cycles of surfaces, Beauville and Voisin~\cite{BV} showed that there exists a "canonical" $0$-cycle $o_S \in \CH_0(S)$ given the class of any point supported on any (singular) rational curve in $S$. Moreover, this $0$-cycle satisfies the following properties: 
\begin{thm}\label{BVmain}
\hfill
\begin{enumerate}[(i)]
\item For any $L, L' \in \Pic(S)$, their intersection product in $\CH_0(S)$ is proportional to $o_S$;
\item The second Chern class of $S$ satisfies the equality  in $\CH_0(S)$:
\begin{equation}\label{maineq}
c_2(S) = 24 \cdot o_S.
\end{equation}
\end{enumerate}
\end{thm}

The first part of Theorem~\ref{BVmain} is mainly a consequence of Bogomolov and Mumford's result~\cite{MM}, saying that every ample complete  linear system $|H|$ in $S$ contains a rational curve. The proof that $c_2(S) = 24 o_S$ in $\CH_0(S)$ is more involved. In~\cite{BV}, Beauville and Voisin proved this identity using a decomposition of the small diagonal $\delta_S = \{(x,x,x) \mid x \in S\}$ in $\CH_2(S \times S \times S)$, viewed as the correspondence  sending $x \in S$ to $\delta_S(x) = (x,x) \in S \times S$. The theorem follows by taking for $\Delta \in \CH_2(S \times S)$ the class of the diagonal, so that  $\delta_S^* \Delta = c_2(S)$. 

Note that the proof of the second statement of the above theorem is rather easy in the case where $S$ is a complete intersection surface in a variety $X$ such that $\CH^\bullet(X)$ is generated by $\CH^1(X)$ as a graded algebra (for example when $X$ is a weighted projective space), in particular for $K3$ surfaces of genus from $2$ to $5$. Indeed, since $c_2(S) = c_2(T_X)_{|S} - c_2(N_{S/X})$ is the restriction of the class of some cycle in $X$, it can be realized as the intersection product of two classes of divisors in $S$ by hypothesis. So $c_2(S)$ is proportional to $o_S$ by the first part of Theorem~\ref{BVmain}.

Another case where this equality is easily proved is the case of a $K3$ surface admitting an elliptic pencil $\pi : S \to \bP^1$ with only reduced components. The proof goes as follows: let $\pi_{\reg} : S_{\reg} \to \bP^1_{\reg} \colonec \pi(S_\reg)$ be the restriction of $\pi$ to the complementary of singular fibers of $\pi$. The restriction of $c_2(S)$ on $S_\reg$ equals $c_1(\pi_\reg^*\gO_{\bP^1_\reg}) \cdot c_1(\gO_{S_\reg/\bP^1_\reg})$ in $\CH_0(S_\reg)$. Thus modulo $0$-cycles supported on singular fibers (which are singular rational curves), $c_2(S)$ is the product of the first Chern class of two line bundles in $\CH_0(S)$ by localization, hence $c_2(S)$ is proportional to $o_S$.

The aim of this paper is to give a more geometric proof in the spirit of the proof given above, of equality~(\ref{maineq}) under some hypotheses on the singularities of the one-parameter family of elliptic curves determined by the polarization, which are expected to be true for a very general polarized $K3$ surface of genus $\ge 2$:
\begin{Def}\label{hypE}
Let $(S,L)$ be a primitively polarized $K3$ surface. We say that $(S,L)$ satisfies condition $(E)$ if 
\begin{enumerate}[i)]
\item $\Pic(S) = \bZ \cdot L$;
\item If $E \in |L|$ is an elliptic curve which has a non-immersive point, the singular points of $E$ consist of ordinary multiple points and a cusp ($A_2$ singularity type).  
\end{enumerate}
\end{Def}

The main Theorem that we will prove is the following, which from the definition of the Beauville-Voisin $0$-cycle is equivalent to the identity $c_2(S) = 24 \cdot o_S$ in the Chow group of $0$-cycles of $S$.
\begin{thm}\label{main}
If $(S,L)$ satisfies condition $(E)$, then $c_2(S) \in \CH_0(S)$ is supported on rational curves.
\end{thm}

Note that since rational curves can only specialize to rational curves, in order to prove the second statement of Theorem~\ref{BVmain} in full generality, it suffices to prove it for a very general $(S,L)$ in the moduli space of polarized $K3$ surfaces. By Theorem~\ref{main}, one can try to prove this by showing that a very general polarized $K3$ surface satisfies condition $(E)$. This is expected to be true in light of the  following two theorems:

\begin{thm}[X. Chen~\cite{XChen}]\label{ratnodal}
All rational curves in the primitive class of a general $K3$ surface of genus $g \ge 2$ are nodal.
\end{thm}

\begin{thm}[Diaz-Harris~{\cite[Theorem $1.4$]{DiazHarris}}]\label{DH}
Let $V_{d,g}$ be the Severi variety parametrizing plane curves of degree $d$ and geometric genus $g$. If $W \subset V_{d,g}$ is any codimension $1$ subvariety and $p \in W$ a general point representing a curve $C$, then the singularities of $C$ are either
\begin{enumerate}[i)]
\item n nodes ($A_1$ singularity type);
\item $n-1$ nodes and one cusp ($A_2$ singularity type);
\item $n-2$ nodes and one tacnode ($A_3$ singularity type);
\item $n-3$ nodes and one ordinary triple point ($D_4$ singularity type),
\end{enumerate}
where $n = \frac{1}{2}(d-1)(d-2) - g$.
\end{thm}


In view of Theorem~\ref{DH}, condition $(E)$ would be the statement that the family of elliptic curves in $|L|$ when $S$ is general and $\Pic(S) = \bZ \cdot L$   is in general position with respect to the stratification of locally planar curves by singularity and this expectation is reasonable in view of Theorem~\ref{ratnodal} which can be seen as a similar general position statement.

By standard deformation-theoretic arguments, we know that for a general $(S,L)$ and any $0\le g \le \frac{L^2}{2} +1$, a general point in the subspace of $|L|$ parametrizing curves of geometric genus $g$ corresponds to a nodal curve. See also~\cite{ GalatiTriple, GalatiKnutsen} for results pointing to the same direction.

Back to Theorem~\ref{main}, the bridge connecting $c_2(S)$ and Beauville-Voisin's $0$-cycle $o_S$ is the locus $Z$ of cusp points in the one-parameter family of elliptic curves (the precise definition will be given in Section~\ref{step1}). Roughly, there are two main steps in the proof of Theorem~\ref{main}:
\begin{enumerate}[i)]
\item Show that $c_2(S)$ is proportional to $[Z]$ in $\CH_0(S)_{\bQ}$ modulo $0$-cycles supported on rational curves;
\item Prove that $[Z]$ is proportional to $o_S$  in $\CH_0(S)_{\bQ} \colonec \CH_0(S) \otimes \bQ$.
\end{enumerate}
Each step will be carried out respectively in Section~\ref{step1} and the last two sections after we prove a deformation-theoretic result of elliptic curves in $S$ in Section~\ref{smoothhom}. We will prove Theorem~\ref{main} in the last section. 

Varieties are defined to be reduced and irreducible. All schemes and varieties are defined over the field of complex numbers $\bC$. The $K$-group of coherent sheaves (resp. locally free sheaves) on $X$ is denoted by $K_0(X)$ (resp. $K^0(X)$).  For the definitions, notations and basic properties about intersection theory we refer to~\cite{Fulton}. 



\section{Universal deformation space of morphisms from elliptic curves to a fix $K3$ surface}\label{smoothhom}



One can  adopt either the viewpoint of Hilbert scheme or the viewpoint of scheme parametrizing morphisms to study deformations of embedding curves. For our purpose, it is more appropriate to study them \emph{via} the latter, since we will be interested in deformations of curves with fixed geometric genus. Indeed, let
\begin{center}
\begin{tikzpicture}
\centering
\matrix (m) [matrix of math nodes, row sep=1.5em,
column sep=1.5em]
{ \cC &  S  \\
B \\
};
\path[->,font=\scriptsize]
(m-1-1) edge node[left] {} (m-1-2)
(m-1-1) edge node[left] {} (m-2-1);
\end{tikzpicture}
\end{center}
be a family of elliptic curves in $S$, then all $0$-cycles supported on (rational) singular fibers are proportional to $o_S$ in $\CH_0(S)$. Hence we can reduce to study
\begin{center}
\begin{tikzpicture}
\centering
\matrix (m) [matrix of math nodes, row sep=1.5em,
column sep=1.5em]
{ \cC^\circ &  S  \\
B^\circ \\
};
\path[->,font=\scriptsize]
(m-1-1) edge node[left] {} (m-1-2)
(m-1-1) edge node[left] {} (m-2-1);
\end{tikzpicture}
\end{center}
with $B^\circ$ parameterizing smooth fibers. 

 Let $(S,L)$ be a projective polarized $K3$ surface of genus $g \ge 2$ and let $\Defo(f:E \to S)$ denote the universal deformation space of the morphism $f$ from a varying smooth elliptic curve $E$ to $S$.


\begin{lem}\label{Homsmooth}
$\Defo(f)$ is smooth at $[f]$ if the non-immersive locus of $f$ is empty or a reduced point $p \in E$. 
\end{lem}

What we call here the non-immersive locus is the subscheme of $E$ defined by the vanishing of the morphism $f_*:T_E \to f^*T_S$.

\begin{proof}
Let $r: H^1(N_f) \to H^2(\cO_S)$ be the semi-regular map defined in~\cite{Ran} of $f$ where $N_f \colonec \coker\(TE \to f^*TS\)$ is the normal sheaf of $f$. By~\cite{Ran}, it suffices to show that $\ker (r) = 0$. 

 We refer to~\cite[Corollary $4$]{Ran} for the case where $f$ is unramified. If the ramification divisor of $f$ is a reduced point $p$, then $TE\otimes \cO_E(p)$ is the saturation of $TE$ in $f^*TS$. Hence $N_f = \cO_E(-p) \oplus i_p(\bC)$ where $i_p(\bC)$ is the skyscraper sheaf supported on $p$. Thus $h^1(N_f) = 1$. Since it is also shown in the proof of Corollary $3$ of~\cite{Ran} that $r$ is a nonzero map, we conclude that $r$ is injective for dimensional reason. 
\end{proof}

As an immediate consequence, 
\begin{cor}
If $(S,L)$ satisfies condition $(E)$, then $\Defo(f)$ is smooth for any $f$ such that $f_*[E] = c_1(L)$.
\end{cor}

Let us now give a local description of the universal morphism of $\Defo(f)$ where $f : E \to S$ is a map whose non-immersive locus is a reduced point $p \in E$.

\begin{lem}
With a suitable choice of local coordinates at $p$, the universal morphism parametrized by $\Defo(f)$ is locally given by 
\begin{equation}\label{coordcusp}
f_\univ: (x,y) \mapsto \(x +y^2,y^3 + \frac{3}{2}xy\),
\end{equation} 
where $x$ is the coordinate parametrizing $\Defo(f)$ and $y$ is a local vertical coordinate for the family of curves $\cC \to \Defo(f)$.
\end{lem}

 For simplicity, any local coordinates in which $f_\univ$ is given by~(\ref{coordcusp}) are called \emph{local coordinates~(\ref{coordcusp})}.

\begin{proof}
First we choose coordinates $(x,y)$ in an analytic neighborhood of $f(p)$ such that $f$ is given by 
$$t \mapsto (t^2,t^3)$$
where $t$ is a local coordinate on $E$ in an analytic neighborhood of $p$. The differential $df$ is given by
$$\frac{\dr}{\dr t} \mapsto t\(2\frac{\dr}{\dr x} + 3t\frac{\dr}{\dr y}\).$$
Denote $\gl(t) \colonec \(2\frac{\dr}{\dr x} + 3t\frac{\dr}{\dr y}\)$; on one hand we have that the torsion subsheaf $N_\tors$ of $N_f$ is generated by the section $\gl$. On the other hand, let $U \subset E$ be an analytic neighborhood of $p$, the normal sheaf of $f_{|U}$ is $(N_f)_{|U} = \cO_{E\cap U}(-p)\oplus i_p(\bC)$ and the induced map $H^0(E,N_f) \to H^0(U,N_f) \to H^0(U, i_p(\bC))$ is an isomorphism. Thus the image of $\frac{\dr}{\dr t}$ under the Kodaira-Spencer-Horikawa map is nonzero and is proportional to $\gl$, and the lemma follows.
\end{proof}



\section{Cusp locus and $c_2(S)$}\label{step1}

As before, let $(S,L)$ be a primitively polarized $K3$ surface. Since there is a one-dimensional family of elliptic curves in $|L|$~\cite{MM}, there exists a (maximally defined) dominant rational map $f : \gS \dto S$ from an elliptic surface $\pi : \gS \to \ol{B}$ over a \emph{smooth} projective curve $\ol{B}$ to $S$ such that  $f$ is well-defined on a general fiber of $\pi$ and the image of these fibers under $f$ are elliptic curves in $|L|$. Let 
$$B \colonec \left\{ x \in \ol{B} \mid f \text{ is well-defined along } \gS_x \colonec \pi^{-1}(x) \text{ and } \gS_x \text{ is a smooth elliptic curve } \right\}$$ 
which is  Zariski open in $\ol{B}$. 

From now on, we assume that $(S,L)$ satisfies condition $(E)$ (\emph{cf.} Definition~\ref{hypE}). 

\begin{lem}\label{isoloc}
For any $b \in B$, $B$ is locally isomorphic to the universal deformation space of $f_{|\gS_x}$ and $f$ its universal morphism.
\end{lem}
\begin{proof}
Since $(S,L)$ satisfies condition $(E)$, applying Lemma~\ref{Homsmooth} to $f_{|\gS_b}$ for each $b \in B$ yields the smoothness of the universal deformation space of $f_{|\gS_b}$ for all $b \in B$. Since the fibration $\pi_{|\gS} : \gS \to B$ is not locally iso-trivial, the lemma is proven.  
\end{proof}

Let $\tilde{f} : \wt{\gS} \to S$  be the minimal resolution of $f$ such that $\wt{\gS}$ is also smooth.  We summarize the situation by the following diagram:
\begin{center}
\begin{tikzpicture}
\centering
\matrix (m) [matrix of math nodes, row sep=1.5em,
column sep=1.5em]
{  & \wt{\gS} &  \\
\gS & &  S \\
\ol{B}  \\};
\path[->,font=\scriptsize]
(m-1-2) edge node[left] {$\tau$} (m-2-1)
edge[bend right = 80] node[left] {$\tilde{\pi}$} (m-3-1)
edge node[auto] {$\tilde{f}$} (m-2-3)
(m-2-1) edge node[auto] {$\pi$} (m-3-1);
\path[densely dashed, ->,font=\scriptsize]
(m-2-1) edge node[auto] {$f$} (m-2-3);
\end{tikzpicture}
\end{center}

Since $\Pic(S) = \bZ$ which is generated by the class $\tilde{f}(\wt{\gS}_b)$ for any $b \in \ol{B}$, $\tilde{f}(\wt{\gS}_b)$ is irreducible for all $b \in \ol{B}$. Thus there are three types of fibers of $\tilde{\pi}$:
\begin{enumerate}[i)]
\item smooth elliptic curves;
\item irreducible rational curves;
\item union of a  elliptic curve contracted by $\tilde{f}$ and a rational curve.
\end{enumerate}
Let $U$ be the complement of type $(ii)$ and $(iii)$ fibers in $\wt{\gS}$. Note that since by definition $U = \pi^{-1}(B)$,  $\tau_{|\tau^{-1}(U)}$ is isomorphic onto its image and  $f$ is well-defined on $U$. For simplicity, we will identify $\tau^{-1}(U)$ with $U$ in the rest of the paper.

\begin{lem} \label{mod}
If $z \in \CH_0(\wt{\Sigma})$ is a $0$-cycle supported on $\wt{\gS} - U$, then $\tilde{f}_*z$ is supported on rational curves and thus proportional to $o_S$. 
\end{lem}
\begin{proof}
Indeed, since the elliptic component in a type $(iii)$ fiber is contracted to a point, the images under $\tilde{f}$ of fibers of type $(ii)$ and $(iii)$ are rational curves. 
\end{proof}

\begin{lem}\label{ramfib}
The ramification divisor of $\tilde{f}$ is supported on fibers of $\tilde{\pi}$. In particular, the restriction of $\tilde{f}$ on a general fiber of $\tilde{\pi}$ is an immersion.
\end{lem}
The last statement of the above lemma can also be deduced easily from the well-known fact that a general elliptic curve in $|L|$ is nodal since $L$ is primitive.

\begin{proof}[Proof of Lemma~\ref{ramfib}]
Since $\wt{\gS}$ and $S$ are smooth, the pullback morphism $\tilde{f}^* : \tilde{f}^*\gO_S \to \gO_{\wt{\gS}}$ is injective. If $E$ is any smooth fiber of $\tilde{\pi}$, one has 
$$\det(\tilde{f}^*\gO_S)_{|E} = \det{\gO_{\wt{\gS}}}_{|E} = \cO_E.$$
So the map $\tilde{f}^*_{|E} : {\tilde{f}^*\gO_S}_{|E} \to {\gO_{\wt{\gS}}}_{|E}$ is either an isomorphism or  everywhere degenerated, which proves the lemma.
\end{proof}

\begin{rem}
 We also deduce that since $\tilde{f}^*c_2(S) = c_2(\wt{\gS}) - c_2(Q)$ where $Q$ is the cokernel of $ {\tilde{f}^*\gO_S} \to {\gO_{\wt{\gS}}}$ (because $c_1(S) = 0$), modulo $0$-cycles supported on rational curves the second Chern class of $S$ is supported on non-immersive elliptic curves in $|L|$ by Remark~\ref{mod}. In fact we will rather use formula~(\ref{ideq}) in Proposition~\ref{id} below which is closely related.
\end{rem}

In order to analysis the part of $c_2(S)$ which is supported on these non-immersive elliptic curves, let us first describe the ramification locus of $f_{|U}$ for $(S,L)$ satisfying condition $(E)$.

\begin{lem}\label{simpram}
If $(S,L)$ satisfies condition $(E)$, then the ramification locus of $f_{|U}$ is reduced, and is supported on fibers of $\pi$ whose image under $f$ is a cuspidal elliptic curve.
\end{lem}

\begin{proof}
Let $E$ be a fiber of $\pi_{|U} : U \to B$ such that $f_{|E}$ is an immersion. Assume that in an analytic neighborhood of  $z \in E$,  $df : TU \to f^*TS$  is given by
$$(\phi_1,\phi_2) : \pi^*TB \times TE \to f^*TS.$$
Since $f_{|E}$ is an immersion, $\phi_2$ is nonzero. Moreover, since the $\Defo(f_{|E})$ is smooth and $f_{|U}$ is the universal morphism of $\Defo(f_{|E})$ in an analytic neighborhood of $E$, ${\phi_1}_{|E} : \pi^*TB_{|E} \to f^*TS$ is the restriction of a nonzero element in $H^0(E,N_{f_{|E}})$, so $\det (\phi_1,\phi_2)$ is not identically zero. Hence  $f$ is unramified at immersive fibers.

If $f_{|E}$ ramifies at $z$, then $f(E)$ is a curve with one cusp. By Lemma~\ref{isoloc} we can work with local coordinates~(\ref{coordcusp}) $(x,y)$ in $U$ at $z \in E$  so that the first projection coincides with $\pi$ and that $f$ is given by
$$(x,y) \mapsto \(x +y^2,y^3 + \frac{3}{2}xy\)$$
in these local coordinates. We immediately verify that the ramification divisor along $E$ is reduced.
\end{proof}

For all $p \in \ol{B}$, let $\gS_p$ denote the fiber of $\pi$ over $p$. Since a general member of the family of elliptic curves determined by $|L|$ is a nodal curve, the points $z \in U$ such that $f(z)$ is a cusp in $f(\gS_{\pi(z)})$ form a discrete set $Z_\cusp \subset U$; we define the \emph{cusp locus} to be the $0$-dimensional reduced subscheme of $U$ whose underlying set is $Z_{\cusp}$.


\begin{pro}\label{id}
The following identity holds in $\CH_0(S) / \bZ o_S$
\begin{equation}\label{ideq}
(\deg f) \cdot c_2(S) = \tilde{f}_* [Z_\mathrm{cusp}].
\end{equation}
\end{pro}

\begin{proof}
Recall that the fibers of $\pi$ in $U$ are smooth elliptic curves, so $c_1(\gO_{U/B}) = \pi^*L$ for some $L \in \Pic(B)$. Since $c_1(\tilde{f}^*TS)_{|U} = 0$, one has 
$$c_2(\tilde{f}^*TS)_{|U} = c_2\(\gO_{U/B} \otimes \(\tilde{f}^*TS\)_{|U}\),$$
which is determined by the scheme-theoretic vanishing locus $V(s)$ of the section $s$ of $\gO_{U/B} \otimes f_{|U}^*TS$ induced by $df : TU \to f_{|U}^*TS$. Suppose that $f$ is given by $(x,y) \mapsto (g(x,y),h(x,y))$ in a local coordinates near $p \in U$ such that the first projection $x \mapsto (x,y)$ coincides with $\pi : \gS\to \ol{B}$. Then the degree  of $c_2(f_{|U}^*T_S)$ supported on $p$ is equal to the vanishing order at $(0,0)$ of $(x,y) \mapsto (\dr_y g, \dr_y h)$, which is $0$ if $f_{\pi(p)} \colonec f_{|\gS_{\pi(p)}}$ is an immersion at $p$.

If $f_{\pi(p)}$ ramifies at $p$, then $f(p)$ is a cusp of $f(\gS_{\pi(p)})$ by hypothesis. Using local coordinates~(\ref{coordcusp}) at $p$, the map $f$ is given by
$$(x,y) \mapsto \(x +y^2,y^3 + \frac{3}{2}xy\).$$ 
This implies that the length of $V(s)$ at $p$ is $1$, so $c_2(f_{|U}^*T_S) = [Z_\cusp]$ in $\CH_0(U)$. Hence by Lemma~\ref{mod}
$$(\deg f) \cdot c_2(S) = \tilde{f}_* \tilde{f}^* c_2(S)= \tilde{f}_*\([Z_\cusp]\)$$
in $\CH_0(S) / \bZ o_S$.

\end{proof}

In the next section, we will give a rigorous definition for other important zero cycles appearing naturally in the geometry of the map $f$. Let us describe them in simple terms: the curves $\tilde{f}(\tilde{\gS}_b)$ are nodal, so each curve $\tilde{\gS}_b$ contains two points over each node of $\tilde{f}(\tilde{\gS}_b)$. These points fill in a curve $C$ which roughly speaking is the double point locus of $\tilde{f}$. There is an involution $i$ on this curve exchanging the two points over the same node. Furthermore, there is a morphism of vector bundles (to be defined more carefully latter on) 
$${T_{\gS/B}}_{|C} \oplus i^*{T_{\gS/B}}_{|C} \to (f^*TS)_{|C}.$$
Obviously, the determinant of this morphism  vanishes at cuspidal points. The vanishing locus of this determinant will be analyzed carefully in the last section.

\section{Double point locus}\label{step2}






In this section, we will first define rigorously the \emph{double point locus} $C$ of $\tilde{f}$ and its involution $\gs$. We would like to consider $Z_\cusp$ as the invariant subscheme of $C$ under the involution map $\gs$ (see Lemma~\ref{invcusp}). However it does not work if this double point locus $C$ is defined in $U$, since the involution $\gs$ is not well-defined everywhere, for instance on the pre-images of multiple points of an elliptic curve in $|L|$.

To this end, we first introduce the surface $U_1$, which is defined as the residual scheme to $\Delta_U$ in $U \times_S U  \subset U \times U$, where $\Delta_U$ is the diagonal of $U \times U$ (we refer to~\cite[Chapter $9$]{Fulton} for the definition and cycle-theoretic properties of residual schemes). Then we will define the double point locus $C$ as being contained in $U_1$ (which is different from the usual definition as in~\cite[Chapter $9$]{Fulton}).

First of all, as a consequence of hypothesis $(E)$, we note that


\begin{lem}\label{smoothdlbeloc}
$U_1$ is a surface which is smooth in an analytic neighborhood of $R \colonec U_1 \cap \Delta_U$.
\end{lem}

\begin{proof}

By definition, $R$ is the ramification locus of $\Delta_U \simeq U \to S$, which by Lemma~\ref{ramfib} is reduced and consists of fibers of $\pi$ whose image under $f$ contains a cusp. Let $E$ be one of these fibers and $z$ be a point in $E$. Suppose first that $f_{|E}$ is unramified at $z$. Since the ramification divisor of $f_{|U}$ is reduced, there exist local coordinates $(x,y)$ of $z \in U$ such that the first projection coincides with $\pi$ and that $f$ is given by $(x,y) \mapsto \(x^2 , y\)$. This is because we can first suppose that locally the ramification locus of $f$ is the line $V(x=0) \subset S$. So if $f$ is locally given by $(x,y) \to \(x\cdot\phi(x,y),\psi(x,y)\)$,  since $\jac{f} = x$ it is easy to see that  $\phi = x\cdot h(x,y)$ with $h(0,0) \ne 0$ and that $d\psi$ is invertible at $(0,0)$. Passing to the local chart of $S$ given by the inverse of $(x,y) \mapsto (x \sqrt{h} , \psi)$ allow to conclude.

Taking another copy of $(x,y)$ gives local coordinates $(x,y,x',y')$ around $(z,z) \in U \times U$, then locally at $(z,z)$, the equations defining $U \times_S U \subset U \times U$  is $(x+x')(x-x')=y -y' =0$. Since the diagonal is given by $x-x' = y-y' = 0$, one deduces that $U_1$ is given by $x+x'=y -y' =0$. Hence $U_1$ is smooth at $(z,z)$.

If $f_{|E}$ ramifies at $z$, using the local coordinates~(\ref{coordcusp}), the equations defining $U \times_S U$ at $(z,z) \in U\times U$ are given by
\begin{equation}\label{pdtlocf}
\begin{cases}
    & x + y^2 = x' +y'^2 \\
    &   y^3 + \frac{3}{2}xy = y'^3 + \frac{3}{2}x'y'
\end{cases}
\end{equation}
The first equation defines a smooth hypersurface in $U \times U$. After eliminating $x'$ by the first equation, the second equation becomes
$$ (y - y')\(y^2 + yy' + y'^2 +\frac{3}{2}\(x-y'(y +y')\) \)= 0.$$
So the residual scheme $U_1$ is locally defined in $\bC^3$ by
$$y^2 + yy' + y'^2 +\frac{3}{2}\(x-y'(y +y')\) = 0.$$
Hence $U_1$ is smooth at $(z,z)$.
\end{proof}

Let  $p_1$ and $p_2$ denote the two projections of $U_1 \subset U \times U$ to each component. For $i = 1,2$, let $\pi_i \colonec \pi \circ p_i$. 
We define  the \emph{double point locus of $\tilde{f}$}, denoted by $C$,  to be the residual scheme  to $R$ in $(\pi_1,\pi_2)^{-1}(\Delta_{B}) = U \times_S U \cap U \times_B U$. The involution  exchanging the two components of $U \times U$ induces an involution $\gs$ on $C$. The invariant subscheme of $C$ under $\gs$ is denoted by $C^\gs$.

\begin{lem}\label{lemFinite}
For $i = 1,2$,
\begin{enumerate}[(i)]
\item $p_i : U_1 \to U$ is a finite morphism;
\item The restriction of $\pi_i : U_1 \to B$ to $C$ is finite.
\end{enumerate} 
\end{lem}
\begin{proof}
For $(i)$, suppose that $p_i$ is not finite, then there is a curve $D$ in $U$ such that $f(D)$ is a point. Since $f$ does not contract fibers of $\pi$, $D$ has to dominate $B$, which contradicts Lemma~\ref{ramfib}.

Next, assume that $C \cap \pi_i^{-1}(b)$ has positive dimension for some $b \in B$. Since $p_i$ is finite by $(i)$, the image of $C \cap \pi_i^{-1}(b)$ under $p_i$ is the curve $E \colonec \wt{\gS}_b$. By definition of $C$, this implies that $f_{|E}$ is of degree $>1$ onto its image and contradicts the hypothesis $\Pic(S) = \bZ \cdot L$.
\end{proof}

\begin{lem}\label{smoothalC}
$U_1$ is smooth along $C$. In particular, $C$ is a Cartier divisor in $U_1$. 
\end{lem}
 
\begin{rem}
That $C$ is a Cartier divisor in $U_1$ is in fact a direct consequence of Lemma~\ref{smoothdlbeloc}. Indeed, this follows from the fact that $(\pi_1,\pi_2)^{-1}(\Delta_{B})$ is a Cartier divisor in $U_1$, and that $U_1$ is smooth along $R$.
\end{rem}

\begin{proof}[Proof of Lemma~\ref{smoothalC}]
Let $(z,z') \in C \subset U_1$. If $z=z'$, then $U_1$ is smooth along $(z,z')$ by Lemma~\ref{smoothdlbeloc}. Suppose that $z \ne z'$; since $(S,L)$ satisfies hypothesis $(E)$,  either $z,z' \in U$ do not lie in the ramification divisor of $f$, or $z,z' \in U$ belong to the ramification divisor of $f$ and the local images of $\gS_{\pi(z)}$ at $z$ and at $z'$  under $f$ (denoted $f(\gamma_z)$ and $f(\gamma_{z'})$) intersect transversally at $f(z)$.
In the first case, since any non-zero element in $H^0(\gS_{\pi(z)},N_f)$ vanishes nowhere on  $\gS_{\pi(z)}$ and $f_{|\gS_{\pi(z)}}$ is an immersion, $f$ defines a local isomorphism at between $(U,z)$ and $(S,f(z))$.
Thus one can choose local coordinates centered at $(z,z') \in U \times U$ and $(f(z),f(z')) \in S \times S$ such that locally, $(f,f) : U \times U \to S \times S$ is given by 
\begin{equation}\label{nonramloc}
 \((x,y),(x',y')\) \mapsto \(x,y,u(x',y'),v(x',y')\),
 \end{equation}
 with $\jac{u,v}_{(0,0)} \ne 0$ and such that the first projections $(x,y) \mapsto x$ and $(x',y') \mapsto x'$ give  local expressions of $\pi : U \to B$. Therefore $U\times_S U$, being locally defined by equations
 $$
 \begin{cases}
    & x  = u(x' ,y') \\
    &   y  = v(x',y'),
\end{cases}
$$
with independent differentials, is smooth at $(z,z')$.

In the second case, we showed in the proof of Lemma~\ref{smoothdlbeloc} that there exist local coordinates $(x,y)$ of $z \in U$ such that the first projection coincides with $\pi$ and that $f$ is given by $(x,y) \mapsto \(x^2 , y\)$. Therefore we can assume that there exist local coordinates $\((x,y),(x',y')\) \in U \times U$ such that $(f,f) : U \times U \to S \times S$ is locally given by
\begin{equation}\label{ramloc}
\((x,y),(x',y')\) \mapsto \(x^2,y,u(x',y'),v(x',y')\),
\end{equation}
and that the first projections $(x,y) \mapsto x$ and $(x',y') \mapsto x'$ give  local expressions of $\pi : U \to B$.
Since $z' \in \gS_{\pi(z)}$ is not an immersive point of $f_{|\gS_{\pi(z)}}$, one has $\(\dr_{y'}u(x',y'),\dr_{y'}v(x',y')\) \ne (0,0)$. Furthermore, since the local images $f(\gamma_z)$ and $f(\gamma_{z'})$ intersect transversally at $f(z)$, we have $\dr_{y'}u(0,0) \ne 0$. Hence we deduce that the local equations
 $$
 \begin{cases}
    & x^2  = u(x' ,y') \\
    &   y  = v(x',y'),
\end{cases}
$$
defining $U \times_S U$ have independent differentials at $(z,z')$, \emph{i.e.} $U \times_S U$ is smooth at $(z,z')$.

\end{proof}

Let $j: C \hookrightarrow U_1$ and $j_i \colonec {p_i} \circ j$ for $i=1,2$. Let $\phi \colonec f \circ j_1 = f \circ j_2$. We summarize the situation by the following commutative diagram: 
\begin{center}
\begin{tikzpicture}
\centering
\matrix (m) [matrix of math nodes, row sep=1.5em,
column sep=1.5em]
{ &  & U &  \\
U_1 & C & & S \\
 & &  U & \\};
\path[->,font=\scriptsize]
(m-2-2) edge node[auto] {$j_1$} (m-1-3)
(m-2-2) edge node[left] {$j_2$} (m-3-3)
(m-1-3) edge node[auto] {$f$} (m-2-4)
(m-3-3) edge node[right] {$f$} (m-2-4)
(m-2-2) edge node[auto] {$\phi$} (m-2-4)
(m-2-1) edge[bend left = 40] node[auto] {$p_1$} (m-1-3)
(m-2-1) edge[bend right = 40] node[below] {$p_2$} (m-3-3);
\path[left hook->,font=\scriptsize]
(m-2-2) edge node[] {$j$} (m-2-1);
\end{tikzpicture}
\end{center}
Note that since $C$ is invariant under $\gs$ and $j_1 = j_2 \circ \gs$, we have $j_1(C) = j_2(C)$.

\begin{lem}\label{invcusp}
For $i = 1,2$, $j_i(C^\gs) = Z_\cusp$.
\end{lem}
This equality is \emph{a priori} set-theoretic; however both sets have natural scheme-theoretic structures which are reduced (the reducedness of $C^\gs$ will be shown in the proof below), so \emph{a posteriori} this is also a scheme-theoretic equality.

\begin{proof}[Proof of Lemma~\ref{invcusp}]
Let $(z,z) \in C^\gs$; we will use the same system of local coordinates $(x,y,x',y')$ of $U \times U$ at $(z,z)$ as in the proof of Lemma~\ref{smoothdlbeloc}. Since $C^\gs \subset R$ as a curve in $U$ is the ramification locus of $f: \Delta_U \simeq U \to S$, $z$ is supported on some fiber $E$ of $\pi$ such that $f(E)$ contains a cusp. If $f_{|E}$ is immersive at $z \in E$, then  $U \times_S U \cap U \times_B U$ is given by $x-x' = y-y' = 0$, which are the same equations defining the diagonal in $\Delta_U$. So $z$ is not an immersion point of $f_{|E}$  since $C$ is residual to $\Delta_U$.  Therefore, $j_i(z,z)$ is supported on $Z_\cusp$. Working in local coordinates~(\ref{coordcusp}), the system of equations that defines $(U \times_S U) \cap (U \times_B U)$ at $(z,z)$ is 
\begin{equation}\label{loc1}
\begin{cases}
    & x =x' \\
    & x + y^2 = x' +y'^2 \\
    &   y^3 + \frac{3}{2}xy = y'^3 + \frac{3}{2}x'y'
\end{cases}
\end{equation}
So the equation defining $C$, being the residual scheme to $R$ in $(U \times_S U) \cap (U \times_B U)$, is
\begin{equation}\label{loc2}
\begin{cases}
    & x =x' \\
    & y = -y' \\
    & \frac{3}{2}x-y^2 = 0
\end{cases}
\end{equation}
and in the coordinates $(x,y)$ the involution $\gs$ acts on $C$ by $(x,y) \mapsto (x,-y)$. So locally $C^\gs$ is the reduced point $(0,0)$, hence $j_i(C^\gs) = Z_\cusp$.

\end{proof}

\begin{lem}\label{cuspram}
The scheme-theoretic intersection $C \cap R$ is $C^\gs$. 
\end{lem}
\begin{proof}
It is clear from the system of local equations~\ref{loc1} and~\ref{loc2} describing $(U \times_S U) \cap (U \times_B U)$ and $C$ in the proof of Lemma~\ref{invcusp}.
\end{proof}


Next we define
\begin{equation}\label{psi}
\psi : j_1^* T_{U/B} \oplus j_2^* T_{U/B} \to \phi_{|C}^*TS
\end{equation}
to be the sum of $j_i^* T_{U/B} \to \phi_{|C}^*TS$ induced by ${T_{U/B}} \to f^*_{|U} TS$. Let $W$ be the Cartier divisor of $C$ defined by the vanishing of $\det{\psi}$. The geometry motivation for introducing $\psi$ and $W$ is the fact that away from the ramification of $\tilde{f}$, the vanishing locus of $\det{\psi}$ consists of pairs $(z,z')$ such that the local branches $f(\gamma_z)$ and $f(\gamma_{z'})$ do not intersect transversally. As we will see in the proof of Proposition~\ref{det}, the length of $W$ over $(z,z')$ is precisely one less the intersection multiplicity of $f(\gamma_z)$ and $f(\gamma_{z'})$. The zero cycle $W$ will play a key role in the conclusion of the proof in the next section.

\section{Proof of Theorem~\ref{main}}\label{proof}
 
Since we are dealing with algebraic cycles on varieties which are not necessarily smooth, it would be convenient to introduce the homomorphism  
$$\tau : K_0(X) \to \CH_*(X)_\bQ$$
constructed in~\cite[Theorem $18.3$]{Fulton} by MacPherson~\cite{MacPh} for any algebraic scheme $X$. (One should think to $\tau$ as extending the map $\cF \mapsto \ch(\cF) \cdot \Td(X)$ defined for smooth $X$.)

The properties of $\tau$ that we need is the following, which can be found in~\cite[Theorem $18.3$]{Fulton}:

\begin{itemize}
\item[$\bullet$] $\tau$ defines a homomorphism of $K^0(X)$-modules in the sense that for any $\ga \in K_0(X)$ and any $\gb \in K^0(X)$, one has
$$\tau (\gb \otimes \ga) = \chchar(\gb) \cdot \tau(\ga);$$
\item[$\bullet$] If $\supp(\cF)$ is a $0$-dimensional subscheme of $X$,  then
$$\tau(\cF) = \sum_{x \in \supp(\cF)} l_x(\cF)[x] \in \CH_0(X)_\bQ,$$
where $l_x(\cF)$ denotes the length of the stalk of $\cF$ over the local ring $\cO_{C,x}$.
\item[$\bullet$] If $f : X \to Y$ is a local complete intersection morphism, then 
$$\tau(f^*\ga) = \mathrm{td}(T_f) \cdot f^*\tau(\ga)$$
for any $\ga \in K_0(Y)$ where $T_f \in K^0(X)$ stands for the virtual tangent bundle.
\end{itemize}

Let $\gO_{C/B}$ be the relative cotangent with respect to ${\pi_1}_{|C}= {\pi_2}_{|C} : C \to B$. Since ${\pi_1}_{|C}$ is a finite morphism between curves by Lemma~\ref{lemFinite}, $\gO_{C/B}$ is supported on a $0$-dimensional subscheme. Therefore
$$\tau(\gO_{C/B}) = \sum_{x \in \supp(\gO_{C/B})} l_x(\gO_{C/B})[x] \in \CH_0(C)_\bQ.$$

Back to the morphism $\psi$ we defined in~(\ref{psi}), we can describe the zero locus of $\det \psi$ as follows:

\begin{pro}\label{det}
Let $W \subset C$ be the zero locus of $\det \psi$, then $[W] \in {\pi_1}_{|C}^*\Pic(B) \subset \CH_0(C)$. Moreover, the following identity holds in $\CH_0(C)_\bQ$
\begin{equation}\label{detcusptac}
\tau(\gO_{C/B}) + [C^\gs] = [W].
\end{equation}
\end{pro}

\begin{proof}
Since $W$ is the scheme-theoretic vanishing locus of $\psi :  j_1^* T_{U/B} \oplus j_2^* T_{U/B} \to \phi_{|C}^*TS$ and $c_1(T_S) = 0$, its cycle class $[W]$ in $\CH_0(C)$ is given by
\begin{eqnarray*}
 & [W]  = c_1( j_1^* \gO_{U/B} \oplus j_2^*\gO_{U/B} ) \\ 
& = j_1^*c_1(  \gO_{U/B}) + j_2^* c_1( \gO_{U/B} ). \\ 
\end{eqnarray*}
The first statement of Proposition~\ref{det} follows from the fact that $c_1( \gO_{U/B}) \in \pi^*\Pic(B)$.

Since $C^\gs$ is reduced, to show that $\tau(\gO_{C/B}) + [C^\gs] = [W]$, it suffices to show that for all closed points $x \in C$, 
$$
\begin{cases}
     l_x(W) = 1+ l_x(\gO_{C/B}) \ \ \ &\text{if } \ x \in C^\gs ; \\
     l_x(W) = l_x(\gO_{C/B}) \ \ \ &\text{otherwise}.
\end{cases}
$$


Suppose that $z = z'$, \emph{i.e.} $(z,z') \in C^\gs$ (that is, we are in the cusp situation leading to vanishing of $\psi$). On one hand, using local coordinates~(\ref{coordcusp}) to give local coordinates $(x,y,x',y')$ of $U \times U$ at $(z,z)$, we have in the proof of Lemma~\ref{invcusp} the system of equations~(\ref{loc2})
\begin{equation}
\begin{cases}
    & x =x' \\
    & y = -y' \\
    & \frac{3}{2}x-y^2 = 0
\end{cases}
\end{equation} 
defines the curve $C$ at $(z,z')$. After eliminating the variables $x'$ and $y'$, locally the curve $C$ is given by  $\frac{3}{2}x-y^2 = 0$. Since $C \to B$ is locally the first projection of $(x,y)$, we deduce that $l_{z,z'}(\gO_{C/B}) =1$   

On the other hand, by simple computations $\det \psi$ is given by
$$(x,y,x',y') \mapsto yy'\(2y'-2y + x' - x\).$$
After eliminating $x',y'$ and dividing out $y-y'$ (since $C$ is residual to $\Delta_U$), $\det \psi$ becomes $(x,y) \mapsto -y^2$. Hence $l_{(z,z')}(W) =2$. 

Now assume that $z \ne z'$. Recall first that $(z,z') \in C$ if and only if there exists an elliptic curve $\wt{\gS}_b$ with $z,z' \in \wt{\gS}_b$ and $f(z)=f(z')$. On the other hand by Lemma~\ref{ramfib}, the ramification of $\tilde{f}$ is a union of curves $\wt{\gS}_b$, so either none of $z,z'$ belongs to the ramification divisor of $\tilde{f}$, or both. 

1) First we suppose that $z,z' \in U$ do not lie in the ramification divisor of $f:U\to S$. In this case, as we showed in the proof of Lemma~\ref{smoothalC}, the local expression of $(f , f) : U\times U \to S \times S$ is given by~(\ref{nonramloc}):
$$(x,y,x',y') \mapsto (x,y,u(x',y'),v(x',y')).$$ 
So the equations defining $C$ are 
$$
\begin{cases}
    & x=x' \\
    & x= u(x',y') \\
    & y= v(x',y'). 
\end{cases}
$$
After eliminating the variables $x$ and $y$ by the first and the third equations, the curve $C$ is defined by the vanishing of $g(x',y') \colonec x' - u(x',y')$ in the local coordinates $(x',y')$. On one hand, since $\pi : U \to B$ is locally given by the first projection $(x',y') \mapsto x'$, we obtain the following isomorphism of $\cO_{C,(z,z')}$-algebras:
 $$(\gO_{C/B})_{(z,z')} \simeq \(\bC[x',y'] / (g, \dr_{y'}g)\)_{(x',y')}.$$
On the other hand, in these local coordinates, up to sign $\det \psi$ can be written as
$$\det \psi(x',y') = \dr_{y'} u(x',y') = \dr_{y'} g(x',y').$$
Hence as $\cO_{C,(z,z')}$-algebras,
$$\cO_{W,(z,z')} \simeq \(\bC[x',y'] / (g, \dr_{y'}g)\)_{(x',y')}.$$
In particular, $ l_{(z,z')}(W) = l_{(z,z')}(\gO_{C/B})$.

2) Still assuming $z \ne z'$, it remains to treat the case where $z,z' \in U$  lie in the ramification divisor of $f:U\to S$. It is easy to see that under these assumptions, $(z,z') \notin W$. Indeed, if $(z,z') \in C$ is a point where $\det \psi$ vanishes, this means that the two spaces $(f_{|E})_*(T_{E,z})$ and $(f_{|E})_*(T_{E,z'})$ generate at most a one-dimensional  subspace of $T_{S,f(z)}$ where $E \subset U$ is the elliptic curve passing through $z,z' \in U$. But since $S$ satisfies condition $(E)$, either $\phi(z,z)$ is the cusp of the elliptic curve $f(\gS_{\pi(z)})$ hence $z = z'$, or $z,z' \in U$ do not lie in the ramification divisor of $f:U\to S$, which all violate our assumptions. Hence we need to show that $\gO_{C/B}$ is not supported on $(z,z')$ to finish the proof.

Using local expressions~(\ref{ramloc}) given in the proof of Lemma~\ref{smoothalC}, the equations defining $C$ are 
$$
\begin{cases}
    & x=x' \\
    & x^2= u(x',y') \\
    & y= v(x',y'). 
\end{cases}
$$
So the curve $C$ is defined by the vanishing of $ x'^2 - u(x',y')$ in the local coordinates $(x',y')$. Since $\pi : U \to B$ is locally given by the first projection $(x',y') \mapsto x'$, we deduce that 
 $$(\gO_{C/B})_{(z,z')} \simeq \(\bC[x',y'] / (x'^2 - u, \dr_{y'}u)\)_{(x',y')}$$
as $\cO_{C,(z,z')}$-algebras. Recall in the proof of Lemma~\ref{smoothalC} that since $f_{|E}(\gamma_z)$ and $f_{|E}(\gamma_{z'})$ intersect transversally, one has $\dr_{y'}u(0,0) \ne 0$. Therefore, $\gO_{C/B}$ is not supported on $(z,z')$.



\end{proof}

\begin{lem}
The following identity holds in $\CH_0(C)_\bQ$:
\begin{equation}\label{cotbase}
\tau(\gO_{C/B}) = c_1(j^*\gO_{U_1}) + j^*[C] - j_1^*c_1({\pi^*}\gO_B).
\end{equation}
\end{lem}

\begin{proof}
First of all, since ${\pi_1^*}_{|C}$ is a finite morphism by Lemma~\ref{lemFinite} and since ${\pi_1^*}_{|C}\gO_B$ is torsion-free, the cotangent sequence 
\begin{center}
\begin{tikzpicture}
\centering
\matrix (m) [matrix of math nodes, row sep=1.5em,
column sep=1.5em,text height=1.5ex, text depth=0.25ex]
{ 0 &{\pi_1^*}_{|C}\gO_B  & \gO_{C} & \gO_{C/B} & 0  \\};
\path[->,font=\scriptsize]
(m-1-1) edge node[auto] {} (m-1-2)
(m-1-2) edge node[left] {} (m-1-3)
(m-1-3) edge node[auto] {} (m-1-4)
(m-1-4) edge node[right] {} (m-1-5);
\end{tikzpicture}
\end{center} 
is exact on the left. On the other hand, since $C$ is a Cartier divisor on the smooth surface $U_1$, the conormal exact sequence
\begin{center}
\begin{tikzpicture}
\centering
\matrix (m) [matrix of math nodes, row sep=1.5em,
column sep=1.5em,text height=1.5ex, text depth=0.25ex]
{ 0 &\cO_C(-C)  & j^*\gO_{U_1} & \gO_{C} & 0  \\};
\path[->,font=\scriptsize]
(m-1-1) edge node[auto] {} (m-1-2)
(m-1-2) edge node[left] {} (m-1-3)
(m-1-3) edge node[auto] {} (m-1-4)
(m-1-4) edge node[right] {} (m-1-5);
\end{tikzpicture}
\end{center}
is also exact on the left. Therefore, if we write $\tau(\cO_C) = [C] + \xi$ for some $\xi \in \CH_0(C)_\bQ$, we have
\begin{align*}
\tau(\gO_{C/B}) & = \tau(j^*\gO_{U_1}) - \tau(\cO_C(-C)) - \tau({\pi_1^*}_{|C}\gO_B) \\
& = \(\ch(j^*\gO_{U_1}) - \ch(\cO_C(-C)) - \ch({\pi_1^*}_{|C}\gO_B) \) \cdot \([C] + \xi\) \\
& = c_1(j^*\gO_{U_1}) + j^*[C] - c_1({\pi_1^*}_{|C}\gO_B),
\end{align*}
which finishes the proof.

\end{proof} 

\begin{lem}
The following equality holds in $\CH_0(C)_\bQ$:
\begin{equation}\label{cotfact}
c_1({\gO_{U_1}}_{|C}) = j_1^*c_1(\gO_{U}) + j_2^*c_1(\gO_{U}) - j^*[R].
\end{equation}
\end{lem}
\begin{proof}
First of all, since $p_1 :U_1 \to U$ is finite and $U$ is smooth, $p_1^*\gO_U \to \gO_{U_1}$ is injective. Next, since $p_2^*\gO_{U/S} = \(\gO_{U\times_S U/U}\)_{|U_1}$ and ${\gO_{U\times_S U}}_{|U_1} = \gO_{U_1}(R)$, one has $p_2^*\gO_{U/S} = \gO_{U_1}(R)/p_1^*\gO_U$.  So 
$$c_1(j^*p_2^*\gO_{U/S}) = c_1({\gO_{U_1}}_{|C}) + j^*[R] - j^*p_1^*c_1(\gO_U)$$ 
in $\CH_0(C)$. Finally since $f^*\gO_S \to \gO_U$ is injective and since $j_2 = p_2 \circ j$ is a local complete intersection morphism (indeed, $j_2$ factorizes through $C \hookrightarrow U \times U \to U$ where the first arrow is the closed embedding of $C$ into $U \times U$ and the second arrow is the projection onto the second factor, which is a smooth morphism), we deduce that in $\CH_0(C)_\bQ$,
$$\mathrm{td}(T_{j_2}) \cdot j_2^*\(\tau(f^*\gO_S)- \tau(\gO_U) + \tau(\gO_{U/S})\) = 0$$
If we write $\mathrm{td}(T_{j_2}) = [C] + \xi'$ for some $\xi' \in \CH_0(C)_\bQ$, then we get 
$$c_1(j_2^*\gO_{U/S}) = j_2^*\(c_1(\gO_U) - f^*c_1(\gO_S)\) = j_2^*c_1(\gO_U)$$ in $CH_0(C)_\bQ$. Hence the identity~(\ref{cotfact}) is proven.
\end{proof}

Recall from the definition of $C$ that $[C] = [(\pi_1,\pi_2)^{-1}(\Delta_{B})] - [R]$; for simplicity, we now define 
\begin{equation}\label{Cprime}
C' \colonec (\pi_1,\pi_2)^{-1}(\Delta_{B}).
\end{equation} 
Combining equalities~(\ref{cotbase}) and~(\ref{cotfact}), we obtain
\begin{equation}\label{exprestau}
\tau(\gO_{C/B}) =  j_1^*c_1(\gO_{U}) + j_2^*c_1(\gO_{U})  - j_1^*c_1({\pi^*}\gO_B) + j^*[C'] - 2j^*[R].
\end{equation}

From now on, we fix a $1$-cycle $\ga \in \CH_1(\wt{\gS})$  such that ${\ga}_{|U} = (j_i)_*[C]$ for $i=1,2$. The following result is the key point of the whole computation.
\begin{lem}\label{decompC}
There exists a $1$-cycle $V \in \CH^1(\wt{\gS})$ supported on fibers of $\pi:\wt{\gS} \to \ol{B}$ such that $\ga = V + \tilde{f}^*L$  in $\CH^1(\wt{\gS})_\bQ$.
\end{lem}

\begin{proof}
Let $U^\circ \subset B$ be the Zariski open set parameterizing nodal elliptic curves in $|L|$.  Since $j_1(C) \cap \gS_x$ is the double point locus (defined in~\cite[Chapter $9$]{Fulton}) of the restriction of $f$ to $\gS_x$, its class in $\CH_0(\gS_x)$ is equal to $\(f^*f_*[\gS_x]\)_{|\gS_x} = \(f^*L\)_{|\gS_x}$ by the double point formula~\cite[Theorem $9.3$]{Fulton} (Since $K_S$ and $K_{\gS_x}$ are trivial). The lemma follows from the "spreading principle"~\cite[Theorem $1.2$]{Voisinchowdecomp}. 
\end{proof}

\begin{cor}\label{cordecompC}
If $D \in \CH_1(\wt{\gS})$ is a $1$-cycle supported on the fibers of $\tilde{\pi} : \wt{\gS} \to \ol{B}$, then $\tilde{f}_*(\ga \cdot D) \in \CH_0(S)_\bQ$ is  proportional to $o_S$, or equivalently $\tilde{f}_*(\ga \cdot D) = 0$ in $\CH_0(S)_\bQ /\bQ o_S$.
$$$$
\end{cor}
\begin{proof}
By Lemma~\ref{decompC}, we have
$$\tilde{f}_*\( \ga \cdot D \) = L \cdot \tilde{f}_*D,$$
because $D \cdot V = 0$ and by the projection formula.
So $\tilde{f}_*\( \ga \cdot D \) = 0$  in $\CH_0(S)_\bQ /\bQ o_S$ by Theorem~\ref{BVmain} $(i)$.
\end{proof}

Using Corollary~\ref{cordecompC}, we will prove the two following lemmata:

\begin{lem}\label{vsupp1}
\hfill
\begin{enumerate}[i)]
\item There exists a $0$-cycle $z \in \CH_0(\wt{\gS})$ whose restriction to $U$ is  $ (j_1)_*j_1^*c_1({\pi^*}\gO_B)$ such that $\tilde{f}_*z = 0$ in $\CH_0(S)_\bQ /\bQ o_S$.
\item There exists a $0$-cycle $z_i \in \CH_0(\wt{\gS})$ for $i=1,2$ whose restriction to $U$ is  $(j_i)_*j_i^*c_1(\gO_{U})$ such that $\tilde{f}_*z_i = 0$ in $\CH_0(S)_\bQ /\bQ o_S$.
\end{enumerate}
\end{lem}

\begin{proof}
Since  ${\ga}_{|U} = (j_1)_*[C]$,  the restriction to $U$ of  $z \colonec \ga \cdot \tilde{\pi}^*c_1(\gO_{\ol{B}}) \in \CH_0(\wt{\gS})$ is $(j_1)_*j_1^*c_1({\pi^*}\gO_B)$. Applying Corollary~\ref{cordecompC} to $D = \tilde{\pi}^*c_1(\gO_{\ol{B}})$ allows to conclude.
The same argument applied to $z_i \colonec \ga \cdot c_1(\gO_{\wt{\gS}})$ proves the lemma for $(j_i)_*j_i^*c_1(\gO_{U})$, since $c_1(\gO_{\wt{\gS}})$ is supported on fibers of $\tilde{\pi} : \wt{\gS} \to \ol{B}$.

\end{proof}

Recall that $C'$ is defined in~(\ref{Cprime}).
\begin{lem}\label{vsupp2}
There exists a $0$-cycle $z' \in \CH_0(\wt{\gS})$ whose restriction to $U$ is  $(j_1)_*j^*\([C']\)$ such that $\tilde{f}_*z' = 0$ in $\CH_0(S)_\bQ /\bQ o_S$.
\end{lem}

\begin{proof}

Let $\pr_1:U \times U \to U$ be the first projection, then the following holds in $\CH_0(U)$:

\begin{align*}
(j_1)_*j^*\([C']\) & = (p_1)_* \([C] \cdot [C'] \) \\
& =  (\pr_1)_*(p_1,p_2)_*\((p_1,p_2)^*(\pi,\pi)^*[\Delta_B] \cdot [C]\) \\
& = (\pr_1)_*\((\pi,\pi)^*[\Delta_B] \cdot (p_1,p_2)_*[C]\).
\end{align*}

Now let $j_U$ be the inclusion map of the diagonal $\Delta_U \hookrightarrow U \times U$, on has
\begin{align*}
(\pr_1)_*\((\pi,\pi)^*[\Delta_B] \cdot (p_1,p_2)_*[C]\) & = {\pr_1}_*  {j_U}_*  \( {j_U}^*(\pi,\pi)^*[\Delta_B] \cdot (\pr_1 \circ j_U)^{-1}_*(j_1)_*[C]\) \\
& = \(\pr_1 \circ {j_U}\)_*  \( {j_U}^*(\pi,\pi)^*[\Delta_B]\) \cdot (j_1)_*[C].
\end{align*}
Since 
\begin{center}
\begin{tikzpicture}
\centering
\matrix (m) [matrix of math nodes, row sep=1.5em,
column sep=1.5em]
{ \Delta_U &  U \times U  & U\\
\Delta_B & B \times B   & B\\};
\path[->,font=\scriptsize]
(m-1-1) edge node[auto] {$j_U$} (m-1-2)
(m-1-2) edge node[auto] {$\pr_1$} (m-1-3)
(m-2-2) edge node[auto] {$\pr_1$} (m-2-3)
(m-1-1) edge node[left] {$\pi$} (m-2-1)
(m-2-1) edge node[auto] {$j_B$} (m-2-2)
(m-1-3) edge node[right] {$\pi$} (m-2-3)
(m-1-2) edge node[right] {$(\pi,\pi)$} (m-2-2);
\end{tikzpicture}
\end{center}
is a Cartesian diagram, the left square shows that ${j_U}^*(\pi,\pi)^*[\Delta_B] = \pi^*b_0$ for some  $b_0 \in \CH_0(\Delta_B)$. So \emph{via} the isomorphisms $\pr_1 \circ {j_U}$ and $\pr_1 \circ {j_B}$ in the above Cartesian diagram, if $b \colonec \(\pr_1 \circ {j_B}\)_*b_0 \in \CH_0(B)$, then
$$\(\pr_1 \circ {j_U}\)_*{j_U}^*(\pi,\pi)^*[\Delta_B] = \pi^*b.$$

Therefore, 
$$(j_1)_*j^*\([C']\) = \(\pr_1 \circ {j_U}\)_*  \( {j_U}^*(\pi,\pi)^*[\Delta_B]\) \cdot (j_1)_*[C]=   \pi^*b \cdot (j_1)_*[C].$$
Let $z' \colonec \ga \cdot \tilde{\pi}^*b$; we conclude by applying Corollary~\ref{cordecompC}  to the $1$-cycle $D = \tilde{\pi}^*b$.

\end{proof}

\begin{cor}\label{vsuppcor}
The following equality modulo $0$-cycles supported on rational curves holds in $\CH_0(S)_\bQ / \bQ o_S$:
\begin{equation}\label{eqnn}
\tilde{f}_*\tau(\gO_{C/B}) = -2\tilde{f}_*\([C]\cdot[R]\) = -2\tilde{f}_*\([C^\gs]\).
\end{equation}
\end{cor}
Before we start the proof, we need to clarify the meaning of the above equalities, which have to be understood using Lemma~\ref{mod} as follows: for any $0$-cycle $\xi \in \CH_0(U)$, we define $\ol{\tilde{f}_*\xi} \in \CH_0(S)_\bQ / \bQ o_S $ to be $\tilde{f}_*\gb$  modulo $\bQ o_S$  for any $0$-cycle $\gb \in \CH_0(\wt{\gS})_\bQ$ such that $\gb_{|U} = \xi$. This definition is independent of the choice of $\gb$  by virtue of Lemma~\ref{mod}, hence the identities in Corollary~\ref{vsuppcor} do make sense.

\begin{proof}[Proof of Corollary~\ref{vsuppcor}]
Note that the second equality results directly from Lemma~\ref{cuspram}.
Since $f \circ j_1 = f \circ j_2$, we deduce from Lemmata~\ref{vsupp1} and~\ref{vsupp2} that
$$\tilde{f}_*(j_1)_* \(j_1^*c_1(\gO_{U}) + j_2^*c_1(\gO_{U})  - j_1^*c_1({\pi^*}\gO_B) + j^*[C']\) = \tilde{f}_*\(z_1+z_2-z+z'\) = 0$$
in  $\CH_0(S)_\bQ / \bQ o_S$. Now the corollary follows easily from equality~(\ref{exprestau}).
\end{proof}

We finish the proof of Theorem~\ref{main} as follows:
\begin{proof}[Proof of Theorem~\ref{main}]
By Proposition~\ref{det}, there exists a $1$-cycle $D$ supported on the fibers of $\tilde{\pi} : \wt{\gS} \to \ol{B}$ such that $[W]$ is the restriction of $\ga \cdot D$ to $U$. Thus by Corollary~\ref{cordecompC}, we have $\tilde{f}_*[W] = 0 $ in $\CH_0(S)_\bQ / \bQ o_S$.  So we deduce from equality~(\ref{detcusptac}) that 
$$\tilde{f}_*(\gO_{C/B}) = -\tilde{f}_*(j_1)_*[C^\gs] \text{ in } \CH_0(S)_\bQ / \bQ o_S.$$
 This identity together with equality~(\ref{eqnn}) give
$$\tilde{f}_*(j_1)_*[C^\gs] = 0 \text{ in } \CH_0(S)_\bQ / \bQ o_S.$$ 
By Lemma~\ref{invcusp},
$$\tilde{f}_*[Z_\cusp] = \tilde{f}_*(j_1)_*[C^\gs] \text{ in } \CH_0(S)_\bQ / \bQ o_S.$$
Finally by Lemma~\ref{id}, we conclude that $c_2(S)$ is proportional to $\tilde{f}_*[Z_\cusp] = 0$ in $\CH_0(S)_\bQ / \bQ o_S$.
\end{proof}

We finish this section by stating the following corollary (already appeared in the proof of Theorem~\ref{main}), which can  be considered either as a consequence of Proposition~\ref{det} and Corollary~\ref{vsuppcor}, or of Theorem~\ref{BVmain} and  Proposition~\ref{id}.
\begin{cor}
Let $(S,L)$ be a primitively polarized $K3$ surface satisfying hypothesis $(E)$. The push-forward under $f$ of the class modulo rational equivalence of the cusp locus $Z_\cusp$ is proportional to the Beauville-Voisin $0$-cycle $o_S$.
\end{cor}


\section*{Acknowledgement}

I wish like to thank Claire Voisin for her kindness in sharing her ideas and for many helpful discussions and suggestions. I also thank Xi Chen for interesting e-mail correspondence related to this work.

\bibliographystyle{plain}
\bibliography{BVcan3}

\end{document}